\documentclass[leqno]{amsart}
\usepackage{amsmath,amssymb,amsthm,mathtools}
\usepackage[colorlinks=true,linkcolor=blue,citecolor=blue,urlcolor=blue]{hyperref}
\newtheorem{theorem}{Theorem}[section]
\newtheorem{proposition}[theorem]{Proposition}
\newtheorem{lemma}[theorem]{Lemma}
\newtheorem{corollary}[theorem]{Corollary}
\theoremstyle{definition}
\newtheorem{definition}[theorem]{Definition}

\newtheorem{remark}[theorem]{Remark}
\theoremstyle{plain}
\numberwithin{equation}{section}

\newenvironment{acknowledgments}{\section*{Acknowledgments}}{}

\newcommand{\PP}{\mathbb P}
\newcommand{\EE}{\mathbb E}
\newcommand{\R}{\mathbb R}

\newcommand{\norm}[1]{\left\lVert #1\right\rVert}
\newcommand{\e}{\mathrm e}

\title[Persistence in Probability for Stochastic Delay Equations]{Persistence in Probability for Stochastic Delay Equations with multiplicative noise}
\author{Mark van den Bosch}
\address{Netherlands Organisation for Applied
Scientific Research (TNO), 2597AK The Hague, The Netherlands}
\email{mark-bosch@hotmail.com}

\author{John Ioannis Stavroulakis}
\address{School of Mathematics, Georgia Institute of Technology, Atlanta, GA 30332, United States of America}
\email{jstavroulakis3@gatech.edu}
\date{}
\keywords{stochastic delay differential equation (SDDE); Mackey--Glass equation; bounded in probability; invariant measure; Lyapunov exponent; stochastic integrating factor; fundamental solution}
\subjclass[2020]{34K50, 37L55}

\usepackage{doi}
\usepackage{xcolor} \usepackage[normalem]{ulem} 
\usepackage{todonotes}

\begin{document}
\begin{abstract}
We study the stochastic delay equation
\begin{equation*}
 dX(t)=[-\gamma X(t)+r f(X(t-\tau))]\,dt+\sigma X(t)\,dW(t),
\end{equation*}
where $\gamma,r,\tau\in(0,\infty)$, $\sigma\in[0,\infty)$, and with $f$  either globally Lipschitz or globally bounded and locally Lipschitz, with $f(0)=0$ and $f(x)\ge0$ for all $x>0$.
 Under the condition
\begin{equation*}
 r\liminf_{x\to 0^{+}}\frac{f(x)}{x}>\gamma+\frac{1}{2}\sigma^2,
\end{equation*}
independent of the time delay $\tau,$
we obtain that every solution $X$ with nonnegative initial history and $X(0)>0$ is bounded away from zero in probability. This result then confirms a recent conjecture of van den Bosch, van Gaans, and Verduyn Lunel, and 
yields a non-trivial
invariant measure for, e.g., the  Mackey--Glass  or   Nicholson's blowflies equation with multiplicative noise. To get to this result, we apply a transformation to the solution process, essentially rewriting the equation using a stochastic integrating factor. A comparison with a linear equation gives  contraction estimates for negative conditional moments and hence boundedness away from zero in probability. The same argument yields persistence whenever the linear comparison equation has a positive first Lyapunov exponent.
\end{abstract}
\maketitle

\section{Introduction}
\label{sec:intro}
\noindent Mackey and Glass introduced in 1977 a delay differential equation to model
haematopoietic feedback \cite{MackeyGlass1977}.  It takes the general form 
\begin{equation}\label{eq:1.1}
 x'(t)=-\gamma x(t)+r f(x(t-\tau)),
\end{equation}
with mortality rate $\gamma>0$, recruitment rate $r>0$, a constant time delay $
\tau>0,$ and a delayed feedback
function $f$.  In the haematopoietic interpretation, $x(t)$ represents the
concentration of mature circulating blood cells, while the delay accounts for
the time required for production and maturation before release into the
bloodstream \cite{Glass1979,MackeyGlass1977}.  On the biologically relevant
half-line $x\ge0$, the standard unimodal Mackey--Glass feedback is given by
\begin{equation}
 f(x)=\frac{x}{1+x^n},\qquad n\ge1,\label{eq:1.2}
\end{equation}
particularly satisfying $f(0)=0$ and $f'(0)=1>0$.  Nicholson's blowflies equation can likewise
be modeled by the feedback $f(x)=x\e^{-cx} $ with $ c>0$
\cite{Gurney1980,Nicholson1954}.

The deterministic Mackey--Glass equation \eqref{eq:1.1}--\eqref{eq:1.2} has been studied extensively for   half a century now.  The
numerical study \cite{Farmer1982} quantified high-dimensional chaotic
attractors and their dependence on the delay.  Global and asymptotic dynamics
for scalar delay equations with feedback were developed from several
complementary viewpoints \cite{GarabKrisztin2011,LizRost2009,RostWu2007}.  For nonautonomous
Mackey--Glass equations with variable or multiple delays, positivity,
persistence, oscillation, and stability results were obtained in
\cite{BerezanskyBraverman2006,BerezanskyBraverman2017,BerezanskyBravermanIdels2012}.  Furthermore, periodic dynamics of the classical
Mackey--Glass equation have been studied analytically in \cite{Krisztin2020}
and by computer-assisted methods in
\cite{BarthaKrisztinVigh2021,vanDenBergGrootheddeLessard2022,SzczelinaZgliczynski2018}.  Recent numerical bifurcation analysis exhibits
period-doubling cascades, bistability, and chaotic regimes
\cite{DuruisseauxHumphries2022}.  For broader accounts of the Mackey--Glass
dynamics in the context of delay differential equations, see
\cite{Smith2011,Walther2020}. Chaotic stochastic resonance in the Mackey–Glass equations with additive noise has been studied in \cite{KojimaSato2026}.

Let $W$ denote a standard Brownian motion. A perturbation of the Mackey--Glass equation by multiplicative noise, or more generally
\begin{equation}
\label{eq:stoch-mg-intro}
 dX(t)=[-\gamma X(t)+r f(X(t-\tau))]\,dt+\sigma X(t)\,dW(t),
\end{equation}
was studied by van den Bosch, van Gaans, and Verduyn Lunel in
\cite{vdBMG2026}; for further stochastic delay models and related biological applications, we refer to, e.g.,
\cite{LiuQiuWang2013,LiuJiang2018,MackeyTyranKaminska2024,MaoYuanZou2005,RaoCastilloChavezKang2018,TianBurrageBurrageCarletti2007}.  They proved that solutions to negative feedback system \eqref{eq:stoch-mg-intro} with positive initial histories remain positive, exist globally, and are
bounded above in probability under mild hypotheses.  They further illustrated that|when $f(0)>0$|solutions are bounded away from zero in probability, which in turn is shown to imply the existence of a (physically relevant)
non-trivial invariant measure.  
Finally, the authors  conjectured that|in  case $f(0)=0$ and $f'(0)>0$|solutions are bounded away from zero in probability  if $rf'(0)>\gamma$   and $\sigma$ is sufficiently small; showing this would again confirm the existence of a non-trivial invariant measure \cite[Thm. 1.1]{vdBMG2026}.

In the present paper we prove this conjecture in the following more general
form.  
Throughout this work, we always consider $\gamma>0$, $r>0$, $\tau>0$ and $\sigma\ge0$ to be constants. Regarding the nonlinear function 
$f:\R\to\R$, we either assume
\begin{itemize}
    \item[(A1)] $f$ is globally Lipschitz, or
    \item[(A2)] $f$ is locally Lipschitz and globally bounded, i.e., $\norm{f}_\infty:=\sup_{x\in\R}|f(x)|<\infty$,
\end{itemize}
and work with the condition
\begin{itemize}
\item[(B)] $f$ satisfies $f(0)=0,  f(x)\ge0$ for all $x>0,
  $ and
\begin{equation}
 \kappa:=\liminf_{x\to0^+}\frac{f(x)}x>0.
\end{equation}
\end{itemize}

\noindent\textbf{Definitions.} A stochastic process $X$ is bounded above in probability (uniformly in time) if
\begin{equation}\label{eq:bndinprob}
 \lim_{R\to\infty}\sup_{t\ge0}\PP\{X(t)\ge R\}=0.
\end{equation}
A stochastic process $X$ is bounded away from zero in probability (uniformly in time) if
\begin{equation}\label{eq:bndawayzeroinprob}
 \lim_{\varepsilon\to 0^{+}}
 \sup_{t\ge0}\PP\{X(t)\le\varepsilon\}=0.
\end{equation}
A  nonnegative global solution $X$ to \eqref{eq:stoch-mg-intro} with $X(0)>0$ is called persistent in probability or said to satisfy the permanence property if $X(t)>0$ for all $t>0$ and condition    \eqref{eq:bndawayzeroinprob} holds.

\

The main results of this paper are Theorems \ref{thm:negative-moment} and \ref{thm:positive-moment}, from which we can deduce boundedness away from zero in probability and boundedness from above in probability, respectively. The   explicit criterion \eqref{eq:condition} follows from the finite-time growth condition in Theorem~\ref{thm:negative-moment}.

\begin{theorem}[Persistence in probability]
\label{thm:main}
In addition to assumptions \textnormal{(A1)} or \textnormal{(A2)} together with \textnormal{(B)}, suppose
\begin{equation}\label{eq:condition}
 r\kappa>\gamma+\frac12\sigma^2.
\end{equation}
Then every solution $X$ of stochastic delay  equation \eqref{eq:stoch-mg-intro}, with  nonnegative random
continuous initial history $\Psi$ satisfying $\Psi(0)>0$ almost surely, exists globally and remains positive almost surely, and is 
bounded away from zero in probability uniformly in time.
\end{theorem}

\begin{remark}
For a constant delay $\tau>0$, the time change \begin{equation}\widetilde X(s):=X(\tau s),\qquad
\widetilde W(s):=\tau^{-1/2}W(\tau s),\end{equation}
solves \eqref{eq:stoch-mg-intro} again with unit delay and
$(\gamma,r,\sigma)$ replaced by $(\tau\gamma,\tau r,\sqrt{\tau}\,\sigma)$.
This illustrates that the condition \eqref{eq:condition} in Theorem~\ref{thm:main} is invariant under rescaling to unit delay and explains why we will set $\tau=1$ in \eqref{eq:stoch-mg-intro} without loss of generality throughout the rest of the paper.
\end{remark}
\begin{remark}\label{remark:positiveLyapunov}
Condition \eqref{eq:condition} is sufficient for persistence in probability|but it is not strict. In Section \ref{sec:numerics} we investigate this numerically and demonstrate that the sharp criterion for persistence,
for $b<r\kappa$ and $\kappa>0,$ corresponds to the linear comparison equation
\begin{equation}\begin{aligned}\label{eq:comparison}
 dG(t)&=[-\gamma G(t)+bG(t-1)]\,dt
          +\sigma G(t)\,dW(t),\\
 G(0)&=1,\quad G(u)=0,\, -1\leq u<0,
 \end{aligned}
\end{equation}
having a positive first Lyapunov exponent $\lambda(b)>0$. In particular, we prove that, if $f(x)\leq \kappa x,$ there is  a unique critical rate $r_c$ satisfying
\begin{equation}
 \lambda(r_{\mathrm c}\kappa)=0,
 \qquad 0<\kappa r_{\mathrm c}\le \gamma+\frac12\sigma^2,
\end{equation}
for which persistence holds for $r>r_{\mathrm c}$ and almost-sure extinction  for
$r<r_{\mathrm c}$. 
\end{remark}
\begin{remark}
    As in the deterministic case \cite{BerezanskyBraverman2006}, it suffices to have $\Psi(0)>0$ to obtain positivity of solutions instead of $\Psi(s)>0$ for all $s\in[-\tau,0]$.  In \cite{vdBMG2026} we assume the latter, as this allows us  to log-transform the equation \eqref{eq:stoch-mg-intro}. Indeed, for $Y(t)=
    \log(X(t))$ we get
    \begin{equation}\label{eq:log}
        dY(t)=\big[-(\gamma+\tfrac12\sigma^2)+re^{-Y(t)}f(e^{Y(t-\tau)})\big]dt+\sigma dW(t),
    \end{equation}
    yielding a problem with additive noise and where 
     initial histories take values in $C([-\tau,0];\mathbb R).$ Intuitively, the sufficient condition \eqref{eq:condition} is not too surprising  in view of the log-transformed equation \eqref{eq:log}.
\end{remark}

In this paper, we rewrite   equation \eqref{eq:stoch-mg-intro} by making use of a stochastic integrating factor, see \eqref{eq:phi}, which is in the spirit of Doss' and Sussmann's reduction of stochastic  differential
equations to differential equations with random coefficients \cite{Doss1977,Sussmann1978}. For comparable transformations that have been used to study  oscillatory behavior and Lyapunov exponents, we refer to Appleby, Buckwar, and Kelly \cite{ApplebyBuckwar2005,ApplebyKelly2004},
and Mohammed and Scheutzow \cite{MohammedScheutzow1996,MohammedScheutzow1997}, respectively.

To obtain boundedness away from zero in probability, we then show that the process $1/X$ is $L^p$-stable: negative moments are essentially used as
Lyapunov-Krasovskii functionals \cite{HaleVerduynLunel1993,Krasovskii1959}. In particular, we bound   $\mathbb E[1/X^{p}]$ using solutions of the  linear comparison equation \eqref{eq:comparison} for appropriate values of $b<r\kappa$. This proves boundedness away from zero in probability. Similar methods go back
to Kolmanovskii, Kushner, Ladde, and Shaikhet
\cite{Kolmanovskii1969,Kushner1968,Ladde1974,Shaikhet1975}; see also Baker and Buckwar \cite{BakerBuckwar2005},
Kolmanovskii and Shaikhet \cite{KolmanovskiiShaikhet1993},
Mackey and Nechaeva \cite{MackeyNechaeva1994,MackeyNechaeva1995},
Mao \cite{Mao1997Razumikhin}, Nechaeva and Khusainov
\cite{NechaevaKhusainov1992Estimates,NechaevaKhusainov1992Perturbations,
NechaevaKhusainov1995}, Tsarkov's monograph \cite{Tsarkov1989}, and the
more recent advancements in \cite{XuLiLongTeng2014}.

\begin{remark}
    This new approach  does not seem to apply to (a log-transformed version of) Wright's equation,
for which the proof of boundedness away from zero in probability is  more technical \cite{vdBWright2026}.  In particular, we cannot achieve a lower comparison such as \eqref{eq:voc-lower}, since this requires the nonlinearity $f$ to be nonnegative.
\end{remark}

For $t\ge0$,  define the segment process by
\begin{equation}
 X_t(\theta):=X(t+\theta),\qquad -\tau\le\theta\le0.
\end{equation}
An invariant probability measure is called \emph{non-trivial} if its pushforward under 
current time evaluation (at $\theta=0$) is supported on $(0,\infty)$; in particular, it
differs from the Dirac mass $\delta_{0}$, representing the zero equilibrium state.

Under  the earlier assumptions, if $f$ is bounded, solutions are bounded above in
probability by Theorem~\ref{thm:positive-moment}, cf. \cite[Prop. 3.5 and Prop. 5.5]{vdBMG2026}.  If one positive
solution is also bounded away from zero in probability, then its
logarithm is bounded in probability. After shifting time by $\tau$,
its initial segment is strictly positive by
Proposition~\ref{prop:wellposed}. The Krylov--Bogoliubov construction
for the logarithmic segment process in
\cite[Thm.~1.1 and its proof]{vdBMG2026} therefore yields an invariant
probability measure on $C([-\tau,0];\R)$. Its push-forward under
$\varphi\mapsto\exp\varphi$ is an invariant probability measure
 on $C([-\tau,0];(0,\infty))$. Combining this with
Theorem~\ref{thm:main} 
gives the
following consequence.

\begin{corollary}[Existence of a non-trivial invariant measure]
\label{cor:invariant-existence}
In addition to the assumptions in Theorem \ref{thm:main}, assume $\|f\|_\infty<\infty.$
Then the segment process induces a non-trivial invariant probability
measure on $C([-\tau,0];(0,\infty))$. 
\end{corollary}

For further background on the Krylov--Bogoliubov construction, we refer to
\cite{vdBInvariant2026,vdBMG2026}. Lastly, observe that condition \eqref{eq:condition} can be replaced with the positive Lyapunov exponent criterion in Remark \ref{remark:positiveLyapunov}.

This paper is organized as follows. Section~\ref{sec:solutions} recalls the
basic properties of solutions to equation \eqref{eq:stoch-mg-intro} and proves
a stochastic integrating factor representation.
Section~\ref{sec:prelim} develops the comparison to a linear equation, and  gives an affine delay comparison, a
H\"older estimate for conditional negative moments of its fundamental solution,
and a finite-time comparison with the nonlinear equation.
Section~\ref{sec:proofs} subsequently iterates its negative conditional moment bounds to prove
persistence and derives the explicit criterion \eqref{eq:condition}. 
We end with Section~\ref{sec:5}, where we provide some numerics and intuition regarding our main results, also relating it to Lyapunov exponents.

\section{Properties and representation of solutions}
\label{sec:solutions}
\noindent 
Throughout this paper we utilize the standard filtered-space framework of stochastic calculus \cite{Durrett2019,KaratzasShreve1991}.  Let
$(\Omega,\mathcal F,(\mathcal F_t)_{t\ge0},\PP)$ be a filtered probability
space satisfying the usual conditions, carrying a standard
$(\mathcal F_t)$-Brownian motion $W$. 
On this filtered probability space, we study the stochastic delay differential equation
\begin{equation}\begin{aligned}
 dX(t)
 &=\left[-\gamma X(t)+r f(X(t-1))\right]\,dt
   +\sigma X(t)\,dW(t),\qquad t\ge0,
\label{eq:main}\\
 X(s)&=\Psi(s),\qquad -1\le s\le0,
\end{aligned}\end{equation}
where the noise term needs to be interpreted in the It\^o-sense.

 Concerning the qualitative study of this equation, we are  basically only interested  in nonnegative initial histories
$\Psi:\Omega\to C([-\tau,0];[0,\infty))$. Regarding the linear comparison equation in Section~\ref{sec:prelim}, we need the more general assumption that $\Psi:[-1,0]\times\Omega\to\mathbb R$ is $\mathcal B([-1,0])\otimes\mathcal F_0$-measurable and has bounded   sample paths; we call this a Borel history with bounded paths. 


We use the following notion of a global strong solution.
\begin{definition}
\label{def:strong-solution}
A global strong solution of \eqref{eq:main} is a stochastic process
$X:[-1,\infty)\times\Omega\to\R$ that agrees with $\Psi$ on $[-1,0]$,
is continuous and adapted on $[0,\infty)$, i.e., $X(t)$ is $\mathcal F_t$-measurable for all $t\geq 0$, and satisfies almost surely
\begin{equation}
 X(t)=\Psi(0)+\int_0^t[-\gamma X(s)+r f(X(s-1))]\,ds
 +\int_0^t\sigma X(s)\,dW(s),\label{eq:strongsol}
\end{equation}
for every $t\ge0$.
\end{definition}
The representation \eqref{eq:voc} below is the delay analogue of the 
idea by Doss and Sussmann of removing multiplicative noise by a pathwise change of variables
\cite{Doss1977,Sussmann1978}.  Related integrating factor transformations for
stochastic delay differential equations are used by Appleby, Buckwar, and Kelly in \cite{ApplebyBuckwar2005,ApplebyKelly2004} and in the form of integral equations by Mohammed and Scheutzow \cite{MohammedScheutzow1996,MohammedScheutzow1997} .  Observe that our  transformation absorbs both the noise $\sigma X(t)\,dW(t)$ and the term
$-\gamma X(t)\,dt$ into an integrating factor.

Pathwise, for $X$
close to zero and $f(X(t-1))/X(t-1)$ close to the value $\kappa>0$, the new dependent variable $Z$ satisfies  approximately
\begin{equation}
 Z'(t)=c(t)Z(t-1),
\end{equation}
where $c(t)>0$ depends on the noise. This new variable $Z$ satisfies a
classical delay equation, whose dynamics are well understood \cite{BirkhoffKotin1966Asymptotic,BirkhoffKotin1966Positive,Myshkis1970}. As for $X$ large, the dynamics are relatively easy \cite{vdBMG2026}. Lemma~\ref{lem:finite-time-comparison} handles the transition by comparing with the linear equation until $X$ reaches a fixed small level, and using \eqref{eq:voc-lower} afterwards.

Finally, see  \cite[Ch.~5]{Mao2008} and
\cite{Mohammed1984,vonRenesseScheutzow2010,StojkovicGaans2011} for more
general existence and uniqueness results. Global
existence, uniqueness, and positivity also follow from \cite{vdBMG2026}; we give a direct argument because
it simultaneously provides the representation  below.
\begin{proposition}
\label{prop:wellposed}
Either suppose \textnormal{(A1)} or \textnormal{(A2)}.
Then, for every  Borel history with bounded paths, the initial value
problem \eqref{eq:main} has a unique global strong solution. Define
\begin{equation}\label{eq:phi}
 \Phi(t):=\exp\left\{\sigma W(t)-\left(\gamma+\frac12\sigma^2\right)t\right\}.
\end{equation}
Then, almost surely, for every $t\ge0$,
\begin{equation}
\label{eq:voc}
 X(t)=\Phi(t)Z(t),\qquad
 Z(t):=\Psi(0)+r\int_0^t\Phi(s)^{-1}f(X(s-1))\,ds.
\end{equation}
The paths of $Z$ are locally absolutely continuous.

Furthermore, assume $f(x)\geq 0$ for $x\geq 0$ and let $\Psi$ be nonnegative almost surely.  Then $Z$ is nondecreasing almost surely.
Moreover, $X(t)\ge0$ for all $t\ge0$, almost surely, with strict positivity on
$\{\Psi(0)>0\}$,  and 
\begin{equation}
\label{eq:voc-lower}
 X(t)\ge X(s)\frac{\Phi(t)}{\Phi(s)},\qquad 0\le s\le t.
\end{equation}
\end{proposition}

\begin{proof}
Under the assumption there exists a unique global strong solution, the assertion follows from parts I and II. In parts III and IV we show, by means of the representation \eqref{eq:voc}, that there exists a unique global strong solution.

\emph{Part I. The pathwise transformation.}
Fix $T>0$ and
put
\begin{equation}
 E_T:=\{z\in C([0,T];\R):z(0)=\Psi(0)\},\qquad
 x_z(u):=
 \begin{cases}
 \Psi(u),&-1\le u<0,\\
 \Phi(u)z(u),&0\le u\le T.
 \end{cases}
\end{equation}
The map $\mathcal T:E_T\to E_T$ given by
\begin{equation}
\label{eq:picard}
 (\mathcal Tz)(t):=\Psi(0)+r\int_0^t\Phi(s)^{-1}f(x_z(s-1))\,ds
\end{equation}
is well-defined,  as the integrand is bounded and Borel measurable on
$[0,T]$. Thus every fixed point $Z=\mathcal TZ$ is absolutely continuous.
It\^o's formula \cite[Thm. 3.3]{KaratzasShreve1991} gives
\begin{equation}
\label{eq:Phi-sde}
 d\Phi(t)=-\gamma\Phi(t)\,dt+\sigma\Phi(t)\,dW(t),\qquad\Phi(0)=1.
\end{equation}
Once adaptedness is verified, the product rule \cite[eq. (3.8)]{KaratzasShreve1991} shows that $X=\Phi Z$
solves \eqref{eq:main}. 
 Conversely, It\^o's formula shows that $\Phi^{-1}X$
 solves \eqref{eq:picard}  for every strong solution $X$.

\emph{Part II. Positivity and the lower comparison.}
Assume  $\Psi\ge0$. 
As $Z$ satisfies
\begin{equation}
 Z'(t)=r\Phi(t)^{-1}f(x_Z(t-1))\ge0\quad\text{for a.e. }t\in[0,T],
\end{equation}
nonnegativity of $f$ on $[0,\infty)$ gives that $Z$ is nonnegative and nondecreasing. Hence
$X(t)=\Phi(t)Z(t)\ge\Phi(t)\Psi(0)\ge0$, showing that $X(t)$ is nonnegative for all $t\geq 0$, and positive on the probability event $\{\Psi(0)>0\}.$ In addition, we have
$\Phi(t)Z(t)\ge\Phi(t)Z(s)$, which results in \eqref{eq:voc-lower}.

\emph{Part III. Pathwise existence and uniqueness.}
For a bounded function $z$ on some interval $[a,b]$, write
$
 \textstyle \|z\|_{[a,b]}:=\sup_{a\le s\le b}|z(s)|. $ Put $P=\|\Phi\|_{[0,T]}$ and $Q=\|\Phi^{-1}\|_{[0,T]}$.
If $f$ is globally Lipschitz, let $L$ be a global Lipschitz constant.
Otherwise, set $F=\norm{f}_\infty$, define
\begin{equation}
 R(t):=|\Psi(0)|+rQFt,
\end{equation}
and restrict $E_T$ to the closed subset \begin{equation} \{z\in C([0,T];\mathbb R):z(0)=\Psi(0),\,|z(t)|\le R(t),\,t\in[0,T]\}.\end{equation}
The map $\mathcal T$ preserves this subset. Take $L$ to be a Lipschitz
constant of $f$ on $[-PR(T),PR(T)]$.
In either case, set $D=rPQL$.

Suppose  the solution has been constructed up to $a\in[0,T]$.
On $[a,b]\subset[0,T]$, use the restriction of the already constructed solution on $[a-1,a]$ as initial condition and set
\begin{equation}
 (\mathcal T_a z)(t):=Z(a)
 +r\int_a^t\Phi(s)^{-1}f(x_z(s-1))\,ds.
\end{equation}
This map preserves the bound $|z(t)|\le R(t)$.
Moreover,
\begin{equation}
 \|\mathcal T_a z-\mathcal T_a w\|_{[a,b]}
 \le rLP\int_a^b\Phi(s)^{-1}\,ds\,
       \|z-w\|_{[a,b]}
 \le D(b-a)\|z-w\|_{[a,b]}.
\end{equation}
Thus it is a contraction whenever $b-a\le h:=1/(2+D)$. 
The same $h$ works throughout $[0,T]$, independently of $a$ and the
previously constructed solution. Banach's fixed-point theorem on a
finite partition of mesh at most $h$ gives existence and uniqueness.
The same estimate proves uniqueness among all continuous fixed points;
in the bounded case every such point satisfies $|Z(t)|\le R(t)$.
Uniqueness gives consistency for different $T$, hence a global solution.

\emph{Part IV. Measurability and adaptedness.}
For completeness, we address the adaptedness of the fixed point $Z=\mathcal TZ. $ Consider the Picard sequence
$Z_0(t)\equiv\Psi(0)$, $Z_{n+1}=\mathcal TZ_n, n\geq 0$. Inductively, assume $Z_n$ is continuous and adapted, as is $Z_0$.
For each $0\le t\le T$, the restriction of $x_{Z_n}$ to $[-1,t]\times\Omega$
is $\mathcal B([-1,t])\otimes\mathcal F_t$-measurable: this follows
from the assumption on $\Psi$ at negative times and from continuity
and adaptedness at nonnegative times. Consequently
\begin{equation}
 H_n(s):=r\Phi(s)^{-1}f(x_{Z_n}(s-1))
\end{equation}
is progressively measurable. Its time integral is continuous and
adapted,  hence so is $Z_{n+1}$. 

We show uniform convergence of the Picard sequence, from which adaptedness of the fixed point follows. 
On the partition $0=t_0<\cdots<t_m=T$, put
$e_{n,j}:=\|Z_n-Z\|_{[0,t_j]}$. Splitting the integral at $t_{j-1}$ in estimate (2.14) gives 
\begin{equation}
 e_{n+1,j}\le DT e_{n,j-1}+\tfrac{D}{2+D} e_{n,j}.
\end{equation}
Since $e_{n,0}=0$ for all $n$, we get $e_{n,1}\to0$ as $n$ to $\infty$. Induction over $j$ similarly gives $e_{n,j}\to0$ as $n\to\infty$.
Thus $Z_n\to Z$ uniformly on $[0,T]$, as $n\to\infty,$ and $Z$ is adapted.
\end{proof}

\section{Preliminary lemmata}
\label{sec:prelim}
\noindent In Section~\ref{sec3.1}, we derive a fundamental solution representation for the linear equation with multiplicative noise. Section~\ref{sec3.2} gives an affine delay comparison estimate, which we exploit in  Section~\ref{sec3.3}; there we establish negative-moment and logarithmic integrability bounds for the fundamental solution.  Section~\ref{sec3.4} transfers these estimates to the nonlinear equation.
\subsection{Fundamental solution representation}\label{sec3.1} Consider, for $b\geq 0,$ the initial value problem
\begin{equation}
    x'(t)=-\gamma x(t)+bx(t-1), \quad x(0)=1, \quad x(u)=0, \,-1\leq u<0.
\end{equation}
The solution to this problem can act as a kernel for solutions to the corresponding linear equation with additive noise \cite{KuchlerMensch1992}.  In the case of multiplicative noise, there is again a fundamental solution representation, yet the kernel is no longer deterministic.

Indeed,  let $G_b$, for any $b\geq 0 $, denote the fundamental solution of the linear equation with multiplicative noise,
\begin{equation}
\begin{aligned}
\label{eq:linear-fundamental}
 dG_b(t)&=[-\gamma G_b(t)+bG_b(t-1)]\,dt
          +\sigma G_b(t)\,dW(t),\\
 G_b(0)&=1,\quad G_b(u)=0,\, -1\leq u<0.
 \end{aligned}
\end{equation}
 In particular,  $G_b$ is continuous on $[0,\infty)$,
$G_b=\Phi$ on $[0,1]$, and $G_b\geq \Phi>0$ for $[1,\infty)$, which follows from the fact that Proposition~\ref{prop:wellposed}, with $rf(x)=bx$, applies; recall that $\Phi$ is defined in equation \eqref{eq:phi}.
For $s\ge0$, define the Brownian motion
$W^{(s)}(t):=W(s+t)-W(s)$, $t\ge0$, adapted to the filtration
$(\mathcal F_{s+t})_{t\ge0}$, and write $\smash{G_b^{(s)}}$ for the fundamental
solution of the same linear equation, now driven by $W^{(s)}$,
\begin{equation}
\label{eq:linear-fundamental2}
\begin{aligned}
 dG_b^{(s)}(t)
 &=[-\gamma G_b^{(s)}(t)+bG_b^{(s)}(t-1)]\,dt
   +\sigma G_b^{(s)}(t)\,dW^{(s)}(t),\\
 G_b^{(s)}(0)
 &=1,\qquad G_b^{(s)}(u)=0,\quad -1\le u<0.
\end{aligned}
\end{equation}
Throughout this work|when comparing $G_b^{(s)}(t)$ from \eqref{eq:linear-fundamental2} with $G(t)$ in \eqref{eq:linear-fundamental}|we compare in a pathwise mannner and thus take the same realization of the Brownian motion $W$; this is crucial for the representation \eqref{eq:linear-history} below.
For $t\ge s$, notice that $G_b^{(s)}(t-s)$ is
$\mathcal F_t$-measurable, and that, pathwise, $\smash{G_b^{(s)}}(t-s)$ is continuous\footnote{While continuity in $t$ is obvious,  continuity in $s$ can be seen through performing the step method.} in both $s$ and $t$, {a fortiori} integrable. 

We are now ready to state the fundamental solution representation lemma. 
A global strong solution
to the inhomogeneous problem \eqref{eq:linear:inhomogeneous}
is defined along similar lines as in Definition~\ref{def:strong-solution}.

\begin{lemma}[Fundamental solution representation]
\label{lem:linear-representation}
Let $\Psi$ be a Borel history with bounded paths,
and suppose $g:[0,\infty)\to\mathbb R$ is progressively measurable with
$\smash{\int_0^{\raisebox{-0.3ex}{$\scriptstyle T$}}}|g(s)|\,ds<\infty$, almost surely, for every $T>0$.
The inhomogeneous initial value problem
\begin{equation}
 \begin{aligned}
 dY(t)&=[-\gamma Y(t)+bY(t-1)+g(t)]\,dt
          +\sigma Y(t)\,dW(t),\qquad t\ge0,\\
 Y(s)&=\Psi(s),\qquad -1\le s\le0,
 \end{aligned}\label{eq:linear:inhomogeneous}
\end{equation}
has a unique global strong solution that can be represented pathwise via
\begin{equation}
\label{eq:linear-history}
\begin{aligned}
 Y(t)={}&G_b(t)\Psi(0)
 +b\int_0^{\min\{t,1\}}G_b^{(s)}(t-s)\Psi(s-1)\,ds\\
 &+\int_0^tG_b^{(s)}(t-s)g(s)\,ds,\qquad t\ge0,
\end{aligned}
\end{equation}
\end{lemma}

\begin{proof}
Recall $\Phi$ defined in \eqref{eq:phi}. For $s,u\ge0$, let 
\begin{equation}
 \Phi^{(s)}(u):=\frac{\Phi(s+u)}{\Phi(s)}.
\end{equation}
Proposition~\ref{prop:wellposed}, applied to the Brownian motion
$W^{(s)}(t)$, gives
\begin{equation}
 \frac{G_b^{(s)}(u)}{\Phi^{(s)}(u)}
 =1+b\int_0^u
       \frac{G_b^{(s)}(w-1)}{\Phi^{(s)}(w)}\,dw.
\end{equation}
Setting $u=t-s$, dividing by $\Phi(s)$, and substituting $v=s+w$
yields, 
\begin{equation}
 \frac{G_b^{(s)}(t-s)}{\Phi(t)}
 =\frac1{\Phi(s)}
 +b\int_s^t
       \frac{G_b^{(s)}(v-s-1)}{\Phi(v)}\,dv, \label{8.08}
\end{equation}
for $t\ge s$.

Define the stochastic process $Y$ on $[0,\infty)$ by \eqref{eq:linear-history}, and set
$Y=\Psi$ on $[-1,0]$. Absorb the prescribed history into the forcing by
putting
\begin{equation}
 F(s):=
 \begin{cases}
 b\Psi(s-1)+g(s),&0\le s<1,\\
 g(s),&s\ge1.
 \end{cases}
\end{equation}
Then \eqref{eq:linear-history} reads
\begin{equation}
 Y(t)=\Psi(0)G_b(t)
      +\int_0^tG_b^{(s)}(t-s)F(s)\,ds.
\end{equation}
Dividing by $\Phi(t)$ and using \eqref{8.08} gives
\begin{equation}
\begin{aligned}
 \frac{Y(t)}{\Phi(t)}
 ={}&\Psi(0)+\int_0^t\frac{F(s)}{\Phi(s)}\,ds
 +b\Psi(0)\int_0^t\frac{G_b(v-1)}{\Phi(v)}\,dv\\
 &+b\int_0^t\int_s^t
       \frac{G_b^{(s)}(v-s-1)F(s)}{\Phi(v)}\,dv\,ds.
\end{aligned}
\end{equation}
On each finite time interval, the kernel and $\Phi^{-1}$ are bounded
pathwise and $F$ is integrable. Fubini therefore gives
\begin{equation}
\begin{aligned}
 \frac{Y(t)}{\Phi(t)}
 ={}&\Psi(0)+\int_0^t\frac{F(v)}{\Phi(v)}\,dv\\
 &+b\int_0^t\frac1{\Phi(v)}
 \left[
   \Psi(0)G_b(v-1)
   +\int_0^vG_b^{(s)}(v-s-1)F(s)\,ds
 \right]dv.
\end{aligned}
\end{equation}
Since $G_b^{(s)}$ vanishes at negative times, the part inside brackets is zero
for $v<1$. For $v\ge1$, its inner integral stops at $s=v-1$,
implying that the part inside brackets equals $Y(v-1)$. Combining this with the definition
of $F$ proves
\begin{equation}
 \frac{Y(t)}{\Phi(t)}
 =\Psi(0)+\int_0^t
       \frac{bY(v-1)+g(v)}{\Phi(v)}\,dv.
\end{equation}
As $Y$ defines a continuous adapted process on
$[0,\infty)$, invoking It\^o's product rule   shows that $Y$ satisfies the
stated stochastic delay equation \eqref{eq:linear:inhomogeneous}. 

Finally, Proposition~\ref{prop:wellposed} implies uniqueness of solutions to the homogeneous problem, and therefore uniqueness to the inhomogeneous problem, because taking the difference of any two possible
solutions  satisfies the homogeneous linear equation with zero
history.  
\end{proof}

\subsection{An affine estimate}\label{sec3.2} The next preliminary lemma will be exploited in Lemma~\ref{lem:holder-negative-moment}.

\begin{lemma}
\label{lem:affine-delay-comparison}
Let $a,b\ge0$ and $d>0$. Suppose
$h:[0,\infty)\to\mathbb R$ is locally absolutely continuous that satisfies the inequality
\begin{equation}
 h'(t)\ge b[h(t-1)-h(t)]+d
 \quad\text{for a.e. }t\ge1,
 \qquad h(s)\ge-as\quad(0\le s\le1).
\end{equation}
Then
\begin{equation}
\label{eq:affine-delay-bound}
 h(t)\ge-a+\frac{d}{1+b}(t-1),\qquad t\ge1.
\end{equation}
\end{lemma}

\begin{proof}
Set  $v(t)=-a+d/(1+b)(t-1)$ and $z(t)=\e^{bt}[h(t)-v(t)]$. Then,
\begin{equation}
 z'(t)\ge b\e^b z(t-1)\quad\text{for }t\ge1,
\end{equation}
with $z\ge0$ on $[0,1]$. 
We have $z\ge0$ everywhere and 
 $h\ge v$. On the inequality that $z$ satisfies, cf. \cite{BirkhoffKotin1966Asymptotic,BirkhoffKotin1966Positive,Myshkis1970}
\end{proof}

\subsection{Negative moments of the fundamental solution}\label{sec3.3}

The following lemma uses linear moment comparisons in the spirit of
Ladde's comparison method \cite[Thm.~3.2]{Ladde1974} and the
first-moment analysis of Mackey and Nechaeva
\cite[Sec.~III.B.1]{MackeyNechaeva1995}. For related moment estimates
and stability criteria, see
\cite{BakerBuckwar2005,Kolmanovskii1969,KolmanovskiiShaikhet1993,
Kushner1968,MackeyNechaeva1994,Mao1997Razumikhin,
NechaevaKhusainov1992Estimates,NechaevaKhusainov1992Perturbations,
NechaevaKhusainov1995,Shaikhet1975,Tsarkov1989,XuLiLongTeng2014}.

\begin{lemma}
\label{lem:holder-negative-moment}
Let $b\ge0$ and $p>0$, and set
\begin{equation}
 c:=\gamma+\frac12\sigma^2,\qquad
 a_p:=c+\frac{p\sigma^2}{2},\qquad
 m_p(t):=\EE[G_b(t)^{-p}],\qquad
 h_p(t):=-\frac1p\log m_p(t).
\end{equation}
The functions $m_p$ and $h_p$ are finite and locally absolutely
continuous on $[0,\infty)$, with $h_p(s)=-a_ps$ for $0\le s\le1$.
For a.e. $t\ge1$,
\begin{equation}
\label{eq:holder-moment-inequality}
 m_p'(t)\le pa_pm_p(t)
       -pb\,m_p(t)^{1+1/p}m_p(t-1)^{-1/p},
\end{equation}
and consequently
\begin{equation}
\label{eq:log-moment-delay}
 h_p'(t)\ge-a_p+b\e^{h_p(t-1)-h_p(t)}
 \ge b[h_p(t-1)-h_p(t)]+b-a_p.
\end{equation}
If $b>a_p$, then
\begin{equation}
\label{eq:holder-negative-bound}
 \EE[G_b(t)^{-p}]
 \le\exp\left\{pa_p-\frac{p(b-a_p)}{1+b}(t-1)\right\},
 \qquad t\ge1.
\end{equation}
\end{lemma}

\begin{proof}
Proposition~\ref{prop:wellposed} applied to \eqref{eq:linear-fundamental}
gives
$G_b\ge\Phi>0$, $G_b=\Phi$ on $[0,1]$, and $G_b/\Phi$ is
nondecreasing. Hence, for $t\ge1$,
\begin{equation}
 U(t):=G_b(t)^{-p}\le\Phi(t)^{-p},\qquad
 0<R(t):=\frac{G_b(t-1)}{G_b(t)}
 \le\frac{\Phi(t-1)}{\Phi(t)}.
\end{equation}
Setting $B_T=\sup_{0\le s\le T}|W(s)|$, the reflection principle gives $\EE\e^{aB_T}<\infty$ for every $a>0$
\cite[Sec.~2.6.A]{KaratzasShreve1991}.
On each finite interval, exponentials of $B_T$ dominate both $U^2$ and $UR$, so their expectations are finite. It\^o's formula therefore gives
\begin{equation}
 dU(t)=\bigl[pa_pU(t)-pbU(t)R(t)\bigr]dt
           -p\sigma U(t)\,dW(t),
\end{equation}
with a square-integrable stochastic integral on bounded intervals.
Taking expectations shows that $m_p$ is locally absolutely continuous
and, for a.e. $t\ge1$,
\begin{equation}\label{eq:323}
 m_p'(t)=pa_pm_p(t)-pb\,\EE[U(t)R(t)].
\end{equation}
Since $U(t-1)=U(t)R(t)^{-p}$, H\"older's inequality|with conjugate
exponents $(p+1)/p$ and $p+1$|gives
\begin{equation}
 \begin{aligned}
 m_p(t)
 &=\EE\left[(U(t)R(t))^{p/(p+1)}U(t-1)^{1/(p+1)}\right]\\&\qquad \qquad \le\bigl(\EE[U(t)R(t)]\bigr)^{p/(p+1)}
       m_p(t-1)^{1/(p+1)}.
 \end{aligned}
\end{equation}
Rearranging and substituting into the moment identity \eqref{eq:323} proves 
\eqref{eq:holder-moment-inequality}. Because
$m_p$ is continuous and strictly positive, $h_p$ is locally absolutely
continuous. Dividing \eqref{eq:holder-moment-inequality} on both sides by $-pm_p(t)$
and using $\e^x\ge1+x$ proves \eqref{eq:log-moment-delay}.
The Gaussian moment formula on $[0,1]$ gives
$m_p(s)=\e^{pa_ps}$, so $h_p(s)=-a_ps$. If $b>a_p$,
Lemma~\ref{lem:affine-delay-comparison}, with $a=a_p$ and $d=b-a_p$,
yields
\begin{equation}
 h_p(t)\ge-a_p+\frac{b-a_p}{1+b}(t-1),\qquad t\ge1.
\end{equation}
Exponentiating proves \eqref{eq:holder-negative-bound}.
\end{proof}

\begin{lemma}[Conditional negative-moment estimate]
\label{lem:conditional-negative-moment}
Let $b\ge0$, $p>0$, and $T\ge0$, and set
\begin{equation}
 c:=\gamma+\frac12\sigma^2,\qquad
 a_p:=c+\frac{p\sigma^2}{2},\qquad
 q:=\EE[G_b(T)^{-p}].
\end{equation}
Then, for each fixed $t\ge0$, almost surely 
\begin{equation}
\label{eq:conditional-negative-moment}
 \EE\!\left[G_b(t+T)^{-p}\mid\mathcal F_t\right]
 \le
 qG_b(t)^{-p}. 
\end{equation}
 In particular, if $b>a_p$, then for $T\ge1$, almost surely
\begin{equation}
\label{eq:conditional-holder-negative-bound}
 \EE\!\left[G_b(t+T)^{-p}\mid\mathcal F_t\right]
 \le
 G_b(t)^{-p}
 \exp\left\{
 pa_p-\frac{p(b-a_p)}{1+b}(T-1)
 \right\}.
\end{equation}
\end{lemma}

\begin{proof}
We apply Lemma \ref{lem:linear-representation} with $g=0$ and $Y(u)=G_b(t+u).$ Nonnegativity of the history $\Psi$ term gives
\begin{equation}
 G_b(t+u)\ge G_b(t)G_b^{(t)}(u),\qquad u\ge0.
\end{equation}
Consequently,
\begin{equation}
 G_b(t+T)^{-p}
 \le
 G_b(t)^{-p}\bigl(G_b^{(t)}(T)\bigr)^{-p}.
\end{equation}
Since $G_b^{(t)}$ depends only on the Brownian increments after time $t$,
it is independent of $\mathcal F_t$ and has the same law as $G_b$.
Therefore
\begin{equation}
\begin{aligned}
 \EE[G_b(t+T)^{-p}\mid\mathcal F_t]
 &\le
 G_b(t)^{-p}
 \EE[
 (G_b^{(t)}(T))^{-p}
 \mid\mathcal F_t
 ]
 =G_b(t)^{-p}\EE[
 (G_b^{(t)}(T))^{-p}
 ],
\end{aligned}
\end{equation}
which proves \eqref{eq:conditional-negative-moment}. The bound
\eqref{eq:conditional-holder-negative-bound} follows from
Lemma~\ref{lem:holder-negative-moment}.
\end{proof}

\begin{lemma}[Finite-time moments and logarithmic integrability]
\label{lem:finite-time-log-moments}
Let $b\ge0$ and $T>0$. For every $r>0$,
\begin{equation}
 \EE\!\left[G_b(T)^r\right],\qquad  \EE[G_b(T)^{-r}]<\infty,
 \qquad
 \EE\sup_{0\le s\le T}
       \left(\frac{\Phi(s)}{\Phi(T)}\right)^r<\infty.
\end{equation}
In particular, $\log G_b(T)\in L^1(\mathbb P)$,
and
\begin{equation}
 \left.\frac{d}{dp}\EE[G_b(T)^{-p}]\right|_{p=0+}
 =-\EE[\log G_b(T)].
\end{equation}
Consequently, if $\EE[\log G_b(T)]>0$, then
$\EE[G_b(T)^{-p}]<1$ for all sufficiently small $p>0$.
\end{lemma}
\begin{proof}
Put $c=\gamma+\sigma^2/2$ and $B_T=\sup_{0\le s\le T}|W(s)|$.
The reflection principle gives $\EE\e^{aB_T}<\infty$ for every $a>0$
\cite[Sec.~2.6.A]{KaratzasShreve1991}. Since $G_b\ge\Phi$, we obtain
\begin{equation}
 \begin{aligned}
 \EE G_b(T)^{-r}&\le\EE\Phi(T)^{-r}<\infty,\\
 \EE\sup_{0\le s\le T}
       \left(\frac{\Phi(s)}{\Phi(T)}\right)^r
 &\le\e^{rcT}\EE\e^{2r\sigma B_T}<\infty.
 \end{aligned}
\end{equation}
Writing $Z(t)=G_b(t)/\Phi(t)$, the integrating factor equation \eqref{eq:voc} becomes
\begin{equation}
Z(t)=1+b\int_1^t
\frac{\Phi(v-1)}{\Phi(v)}Z(v-1)\,dv,\qquad t\ge1,
\end{equation}
with $Z(t)=1$ on $[0,1]$. For $B_T=\sup_{0\le u\le T}|W(u)|$, we find
\begin{equation}
\frac{\Phi(v-1)}{\Phi(v)}
=e^{c+\sigma(W(v-1)-W(v))}
\le e^{c+2\sigma B_T},\qquad 1\le v\le T.
\end{equation}
Set $M_j=\sup_{0\le t\le\min\{j,T\}}Z(t)$, so that $M_0=1$.
Since $Z\ge1$,   we get
\begin{equation}
M_j\le 1+bT e^{c+2\sigma B_T}M_{j-1}
\le \bigl(1+bT e^{c+2\sigma B_T}\bigr)M_{j-1}.
\end{equation}
From bootstrapping, and taking into account $\Phi(t)\le e^{\sigma B_T}$, we deduce
\begin{equation}
\sup_{0\le s\le T}G_b(s)
\le e^{\sigma B_T}
\bigl(1+bT e^{c+2\sigma B_T}\bigr)^{\lceil T\rceil}.
\end{equation}
Hence $\EE[G_b(T)^r]<\infty$ for every $r>0$.
Since $|\log x|\le(x^r+x^{-r})/r$, this implies
$\log G_b(T)\in L^1(\mathbb P)$. Moreover, for fixed $p_0,\eta>0$ and
$0<p\le p_0$, the mean-value theorem gives
\begin{equation}
 \left|\frac{x^{-p}-1}{p}\right|
 \le|\log x|\max\{1,x^{-p_0}\}
 \le\eta^{-1}\bigl(x^\eta+x^{-(p_0+\eta)}\bigr),\qquad x>0.
\end{equation}
For $x=G_b(T)$, the last expression is integrable and the difference
quotient converges almost surely to $-\log G_b(T)$ as $p\downarrow0$.
Dominated convergence \cite[Thm.~1.6.7]{Durrett2019} proves the derivative
identity. Clearly $\EE[G_b(T)^0]=1$, and the final assertion readily follows.
\end{proof}

	\subsection{Finite-time comparison with the nonlinear equation}\label{sec3.4}

The  lemma below transfers a finite-time negative-moment estimate for
$G_b$ to the nonlinear equation \eqref{eq:main}, with the same exponent $p$.

\begin{lemma}
\label{lem:finite-time-comparison}
Assume  hypotheses \textnormal{(A1)} or \textnormal{(A2)}, and \textnormal{(B)}. Let $X$ solve the nonlinear equation \eqref{eq:main} with nonnegative
initial history $\Psi$, and assume $\Psi(0)>0$ almost surely. Let $p,b,\delta,T>0$ satisfy
$rf(x)\ge bx$ for $0\le x\le\delta$, and let $G_b$ be as in
\eqref{eq:linear-fundamental}. Define
\begin{equation}
\label{eq:block-constants}
 q:=\EE[G_b(T)^{-p}],\qquad
 D:=\delta^{-p}\EE\sup_{0\le s\le T}
       \left(\frac{\Phi(s)}{\Phi(T)}\right)^p.
\end{equation}
Then $q,D<\infty$, and, for each fixed $t\ge0$, almost surely, 
\begin{equation}
\label{eq:block-contraction}
 \EE[X(t+T)^{-p}\mid\mathcal F_t]\le qX(t)^{-p}+D.
\end{equation}
In particular, this is a contraction estimate whenever $q<1$.
If $\EE\log G_b(T)>0$, then $q<1$ for all sufficiently small $p>0$.
\end{lemma}

\begin{proof}
Lemma~\ref{lem:finite-time-log-moments} gives $q,D<\infty$ and, when
$\EE\log G_b(T)>0$, also $q<1$ for all sufficiently small $p>0$.
Next, we prove \eqref{eq:block-contraction} for every $p>0$, without any smallness
restriction.
To do so,  fix time $t\ge0$, 
write $\Phi^{(t)}(u)=\Phi(t+u)/\Phi(t)$ and let
 $\zeta:=\inf\{u\ge0:X(t+u)\ge\delta\}$, with $\inf\varnothing=\infty$.

For the sample subspace $\{\zeta>T\}$, we perform a routine transformation encoding the initial history as forcing, cf. \cite{AgarwalBerezanskyBravermanDomoshnitsky2012,AzbelevRakhmatullina1972}.
Rewrite  equation \eqref{eq:main} starting at time $t$ with zero initial history, initial value
$X(t)$, and forcing
\begin{equation}
 g(v):=rf(X(t+v-1))-b\mathbf1_{\{v\ge1\}}X(t+v-1).
\end{equation}
This forcing is nonnegative for $0\le v\le T$ on $\{\zeta>T\}$. Indeed, for $v<1$ use $f\ge0$, and for $v\ge1$ use the local lower bound
$rf(x)\ge bx$. Formula~\eqref{eq:linear-history} therefore gives
\begin{equation}
 X(t+u)=X(t)G_b^{(t)}(u)
 +\int_0^uG_b^{(t+v)}(u-v)g(v)\,dv
 \ge X(t)G_b^{(t)}(u),\qquad 0\le u\le T.
\end{equation}
On $\{\zeta\le T\}$, $X(t+\zeta)\ge\delta$ and
\eqref{eq:voc-lower} gives
$X(t+T)\ge\delta\Phi^{(t)}(T)/\Phi^{(t)}(\zeta)$.
Thus, pathwise,
\begin{equation}
 X(t+T)^{-p}\le X(t)^{-p}G_b^{(t)}(T)^{-p}
 +\delta^{-p}\sup_{0\le u\le T}
       \left(\frac{\Phi^{(t)}(u)}{\Phi^{(t)}(T)}\right)^p.
\end{equation}
The Brownian increments after $t$ are independent of $\mathcal F_t$, so
conditioning proves \eqref{eq:block-contraction}.
\end{proof}

\section{Main results}
\label{sec:proofs}
\noindent We are  now ready to prove the main results of this paper. Although the conditional moments depend
on the entire initial history, the bounds below depend on it only through
$\Psi(0)$. For the boundedness away in probability, we require the following additional hypothesis.
\begin{itemize}
    \item[(H)] There exist $b,\delta,T>0$ such that
$rf(x)\ge bx$ for $0\le x\le\delta$ and the fundamental solution
\eqref{eq:linear-fundamental} satisfies $\EE\log G_b(T)>0$.
\end{itemize}

\begin{theorem}
\label{thm:negative-moment}
Assume  hypotheses \textnormal{(A1)} or \textnormal{(A2)}, and \textnormal{(B)}.  Let $X$ solve the nonlinear equation \eqref{eq:main}
with a nonnegative initial history $\Psi$, and assume $\Psi(0)>0$
almost surely. Furthermore, assume hypothesis \textnormal{(H)}. Fix  any
$p>0$ for which $q<1$ holds\footnote{We can take $p>0$ such that $q<1$ by Lemma \ref{lem:finite-time-comparison}}, where both $q$ and $D$ are as in equation
\eqref{eq:block-constants}. Set
\begin{equation}
 a:=\gamma+\frac{p+1}{2}\sigma^2,\qquad
 M:=\e^{paT}\Psi(0)^{-p},\qquad
 C:=\frac{\e^{paT}D}{1-q}.
\end{equation}
For each fixed $t\ge0$, almost surely,
\begin{equation}
\label{eq:main-negative-decay}
 \EE[X(t)^{-p}\mid\mathcal F_0]
 \le C+q^{\lfloor t/T\rfloor}(M-C),
\end{equation}
and consequently
\begin{equation}
\label{eq:negative-uniform}
 \EE[X(t)^{-p}\mid\mathcal F_0]\le\max\{C,M\}.
\end{equation}
For every $R\ge0$,
\begin{align}
\label{eq:lower-tail-uniform}
 \sup_{t\ge0}\PP\{X(t)\le\e^{-R}\}
 &\le\EE\min\{1,\e^{-pR}\max\{C,M\}\},\\
\label{eq:lower-tail-asymptotic}
 \limsup_{t\to\infty}\PP\{X(t)\le\e^{-R}\}
 &\le\min\{1,C\e^{-pR}\}.
\end{align}
In particular, $X$ is bounded away from zero in probability uniformly
in time. 
\end{theorem}

\begin{proof}
Iterating \eqref{eq:block-contraction} and using the tower property gives
\begin{equation}
 \EE[X(nT)^{-p}\mid\mathcal F_0]
 \le q^n\Psi(0)^{-p}+\frac{D(1-q^n)}{1-q},\qquad n\ge0.
\end{equation}
For $t=nT+s$, $0\le s<T$, inequality \eqref{eq:voc-lower} and the Gaussian
moment formula give
\begin{equation}
 \EE[X(t)^{-p}\mid\mathcal F_0]
 \le\e^{pas}\EE[X(nT)^{-p}\mid\mathcal F_0]
 \le C+q^n(M-C).
\end{equation}
This proves
\eqref{eq:main-negative-decay} and \eqref{eq:negative-uniform}.
Then the conditional Markov's inequality yields
\begin{equation}
 \PP\{X(t)\le\e^{-R}\mid\mathcal F_0\}
 \le\min\{1,\e^{-pR}\EE[X(t)^{-p}\mid\mathcal F_0]\}.
\end{equation}
Taking expectations proves \eqref{eq:lower-tail-uniform}; using
\eqref{eq:main-negative-decay} and dominated convergence proves
\eqref{eq:lower-tail-asymptotic}. Since $M<\infty$ almost surely,
dominated convergence as $R\to\infty$ gives  boundedness away
from zero in probability, which finishes the proof.
\end{proof}

\begin{remark}[Continuous-time exponential bound]
\label{rem:continuous-negative-decay}
Estimate \eqref{eq:main-negative-decay} also gives an exponential
bound in continuous time. Indeed, setting
$\lambda_p:=-T^{-1}\log q>0$, one has
$q^{\lfloor t/T\rfloor}\le q^{-1}\e^{-\lambda_p t}$, and hence
\begin{equation}
 \EE[X(t)^{-p}\mid\mathcal F_0]
 \le C+q^{-1}\e^{-\lambda_p t}\max\{0,(M-C)\},
 \qquad t\ge0.
\end{equation}
\end{remark}
\begin{remark}
The local lower bound  $rf(x)\ge bx$
on $[0,\delta]$, together with $f\ge0$, also gives the global estimate
\begin{equation}
 rf(x)\ge bx-b\delta^{-p}x^{p+1},\qquad x\ge0,\quad p>0.
\end{equation}
Indeed, for $x\le\delta$, one applies the local bound, while for $x\ge\delta$, the
right-hand side is clearly nonpositive. Observe that one could have proved a version of \eqref{eq:holder-moment-inequality} for the negative moments of $X$ more directly with the help of this global estimate. Nevertheless, the comparison with the linear equation above handles
the same transition by stopping at $\delta$ and using
\eqref{eq:voc-lower} afterwards. 
\end{remark}

Recall that boundedness above in probability was proved earlier in \cite[Prop.~3.5, Prop.~5.5]{vdBMG2026}. For completeness, we give a conditional-moment proof parallel to Theorem~\ref{thm:negative-moment}.

\begin{theorem}
\label{thm:positive-moment}
Assume  hypotheses \textnormal{(A1)} with $\norm{f}_\infty<\infty$ or \textnormal{(A2)}, and \textnormal{(B)}.   Let $X$ solve the nonlinear equation \eqref{eq:main} with nonnegative
initial history $\Psi$. Let $q\ge1$ and set
\begin{equation}
 c:=\gamma-\frac{q-1}{2}\sigma^2.\label{eq:c}
\end{equation}
Assume $c>0$ holds.  Then, for each fixed $t\ge0$, almost surely,
\begin{equation}
\label{eq:main-positive-decay}
 \left(\EE\!\left[X(t)^q\mid\mathcal F_0\right]\right)^{1/q}
 \le \e^{-ct}\Psi(0)
 +\frac{r\norm{f}_\infty}{c}
 (1-\e^{-ct}).
\end{equation}
and consequently
\begin{equation}
\label{eq:positive-uniform}
 \EE\!\left[X(t)^q\mid\mathcal F_0\right]
 \le
 \max\left\{\Psi(0),
 \frac{r\norm{f}_\infty}{c}\right\}^{q}.
\end{equation}
For every $R\ge0$,
\begin{align}
\label{eq:upper-tail-uniform}
 \sup_{t\ge0}\PP\{X(t)\ge\e^R\}
 &\le
 \EE\min\left\{1,\e^{-qR}
 \max\left\{\Psi(0),
 \frac{r\norm{f}_\infty}{c}
 \right\}^{q}\right\},\\
\label{eq:upper-tail-asymptotic}
 \limsup_{t\to\infty}\PP\{X(t)\ge\e^R\}
 &\le
 \min\left\{1,
 \left(\frac{r\norm{f}_\infty}{c}\right)^q
 \e^{-qR}\right\}.
\end{align}
In particular, $X$ is bounded above in probability uniformly in time.
\end{theorem}

\begin{proof}
Fix $t\ge0$. For $q\ge1$, write
\begin{equation}
 \|H\|_{q,0}:=
 \left(\EE\!\left[|H|^q\mid\mathcal F_0\right]\right)^{1/q}.
\end{equation}
From representation
\eqref{eq:voc}, the conditional
Minkowski integral inequality, and $0\le f\le\norm{f}_\infty$ on
$[0,\infty)$, we obtain the estimate
\begin{equation}
 \|X(t)\|_{q,0}
 \le \Psi(0)\|\Phi(t)\|_{q,0}
 +r\norm{f}_\infty
 \int_0^t\left\|\frac{\Phi(t)}{\Phi(s)}\right\|_{q,0}\,ds.
\end{equation}
The tower property and the Brownian-increment law|together with the Gaussian
moment-generating formula|yields
\begin{equation}
 \left\|\frac{\Phi(t)}{\Phi(s)}\right\|_{q,0}
 =\e^{-c(t-s)},
 \qquad
 \|\Phi(t)\|_{q,0}=\e^{-ct},
\end{equation}
with $c>0$ as in \eqref{eq:c}.
Hence
\begin{equation}
 \left(\EE\!\left[X(t)^q\mid\mathcal F_0\right]\right)^{1/q}
 \le \e^{-ct}\Psi(0)
 +\frac{r\norm{f}_\infty}{c}
 (1-\e^{-ct}),
\end{equation}
which is \eqref{eq:main-positive-decay}.  Observe that the right-hand side is a convex
combination of both $\Psi(0)$ and $r\norm{f}_\infty/c$; this proves
\eqref{eq:positive-uniform}.

For each fixed $t\ge0$, the conditional Markov's inequality gives, almost surely,
\begin{equation}
 \PP\{X(t)\ge\e^R\mid\mathcal F_0\}
 \le \e^{-qR}\EE\!\left[X(t)^q\mid\mathcal F_0\right].
\end{equation}
Insert \eqref{eq:positive-uniform}, truncate the bound by $1$, and take
expectations.  This gives \eqref{eq:upper-tail-uniform}.  Exploiting the inequality
\eqref{eq:main-positive-decay} before taking expectations and then letting
$t\to\infty$ gives
\eqref{eq:upper-tail-asymptotic}, by the dominated convergence  theorem after truncation by $1$.
\end{proof}

\begin{proof}[Proof of Theorem~\ref{thm:main}]
Put $c=\gamma+\sigma^2/2$. Fix any $p>0$ satisfying 
\begin{equation}
 a_p:=c+\frac{p\sigma^2}{2}<r\kappa,
\end{equation}
which is possible under \eqref{eq:condition}, and choose
$a_p<b<r\kappa$. By Lemma~\ref{lem:holder-negative-moment},
\begin{equation}
 \EE[G_b(t)^{-p}]
 \le\exp\left\{pa_p-\frac{p(b-a_p)}{1+b}(t-1)\right\},
 \qquad t\ge1.
\end{equation}
Choose
\begin{equation}
 T>1+\frac{a_p(1+b)}{b-a_p}.
\end{equation}
Then $q:=\EE[G_b(T)^{-p}]<1$, and Jensen's inequality gives
\begin{equation}
 \EE\log G_b(T)
 \ge-\frac1p\log\EE[G_b(T)^{-p}]>0.
\end{equation}
Since $b<r\kappa$, there is $\delta>0$ such that $rf(x)\ge bx$
on $[0,\delta]$. Lemma~\ref{lem:finite-time-comparison} supplies
\eqref{eq:block-contraction} with this same $p$, so
Theorem~\ref{thm:negative-moment} gives the conditional negative-moment
bounds and boundedness away from zero in probability.
Global existence and positivity follow from
Proposition~\ref{prop:wellposed}.
\end{proof}

\section{Numerics and interpretation of main results}
\label{sec:5}
\noindent We begin in Section \ref{sec:numerics} with numerical simulations that illustrate our main results on boundedness in probability and the existence of invariant measures. In Section \ref{sec:history}, we particularly discuss the role of conditioning on $\mathcal F_0$ in our main results. Ultimately, we  investigate in Section \ref{sec:sharpness} the sharpness of the persistence condition \eqref{eq:condition},  by relating the nonlinear dynamics to the first Lyapunov exponent of the linear comparison equation \eqref{eq:linear-fundamental}.

\subsection{Numerics}\label{sec:numerics} Figures \ref{fig:sigma:tiny}--\ref{fig:sigma>1} show  simulations of the log-transformed equation \eqref{eq:log}. Note  that taking the exponential of a solution yields a solution to \eqref{eq:stoch-mg-intro}. In particular, we fix the parameters  $\gamma=0.5$ and $r=1$, increase the noise parameters $\sigma $ throughout the figures, and take the Mackey-Glass functional \eqref{eq:1.2} with $n=30.$ 
For more details on how the numerics are performed, we refer to \cite{vdBMG2026,vdBWright2026}. 

\begin{figure}[!t]
    \centering
    \includegraphics[width=0.99\linewidth, trim = 0 0.9cm 0 0cm, clip]{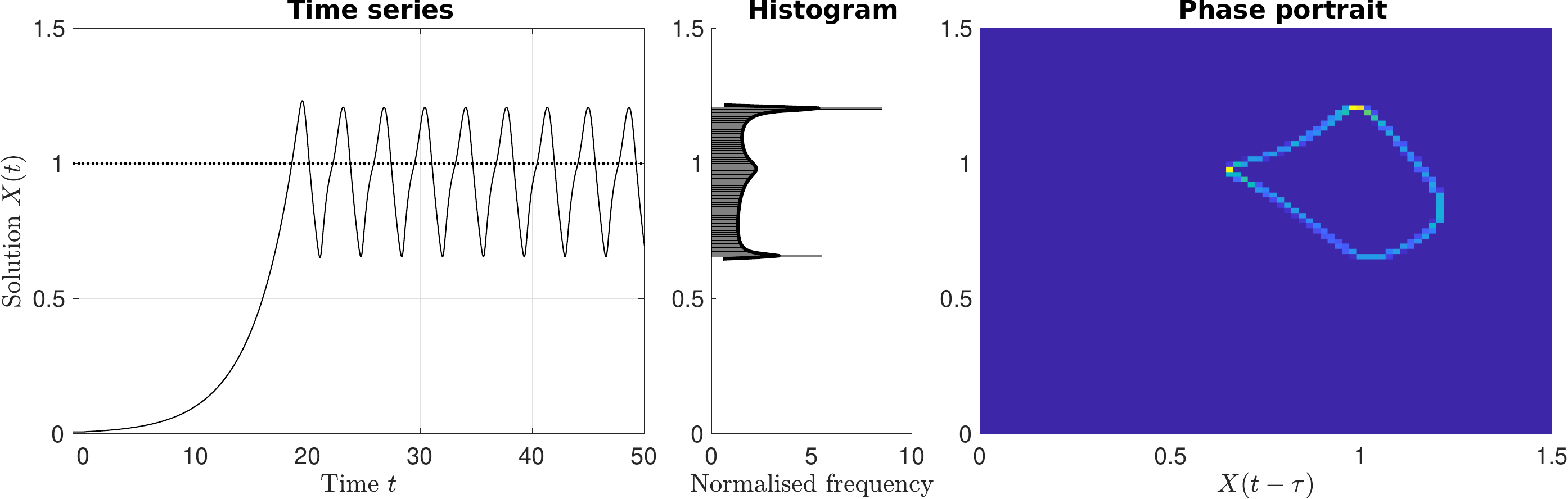}
    \includegraphics[width=0.99\linewidth]{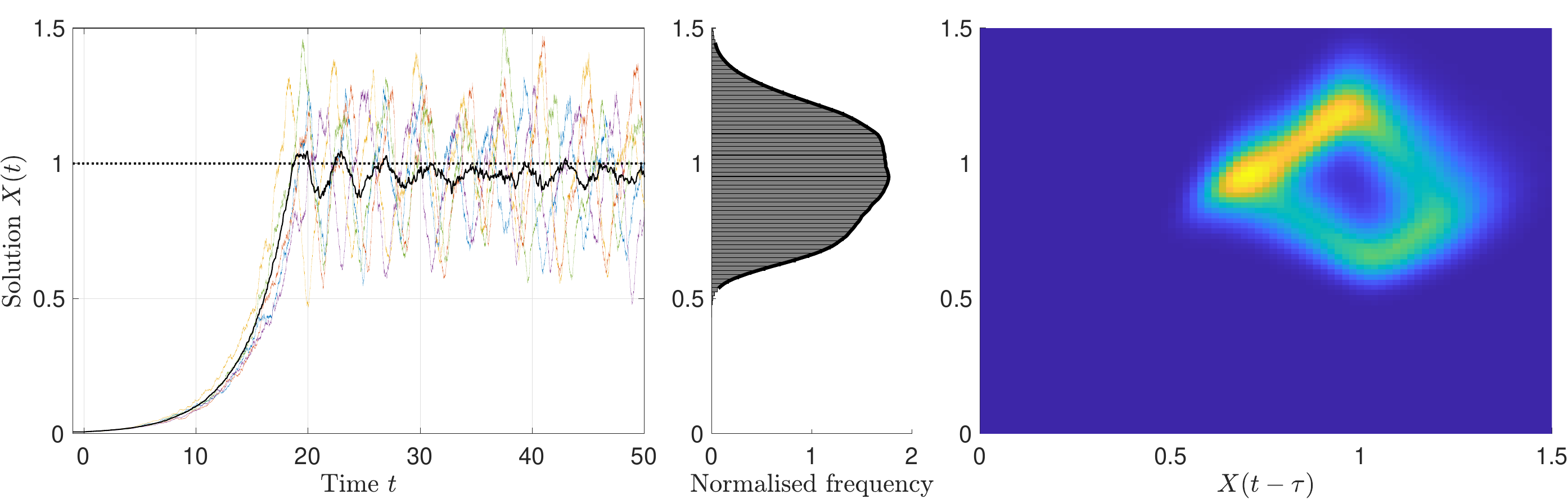}

\caption{Simulation of  100 sample paths of the solution to  equation \eqref{eq:stoch-mg-intro}  on the time interval $[0,5000]$, together with the parameters $\gamma=0.5,$ $r=1,$ $n=30$, 
$f(x)=x(1+x^n)^{-1}$, and $\sigma=0$ (first row, deterministic) or $\sigma=0.1$ (second row). The initial datum is deterministic, $\Psi=\psi$ with $\psi=\exp(-5)$. The left panels show a subset of 10 trajectories on $[-1,50]$; the dotted line $x_*=1$ represents the fixed point of the deterministic equation, while the black curves illustrate the median  of  the 100 sample paths. The middle panels display histograms of all values observed on $[4000,5000]$, to exclude  transient behaviour of  the initial condition. The right panels show heat maps of the pair $(X(t-\tau),X(t)),$ again for all $t\in [4000,5000]$. }
    \label{fig:sigma:tiny}
\end{figure}

\begin{figure}[!t]
    \centering
    \includegraphics[width=0.99\linewidth, trim = 0 0.9cm 0 0cm, clip]{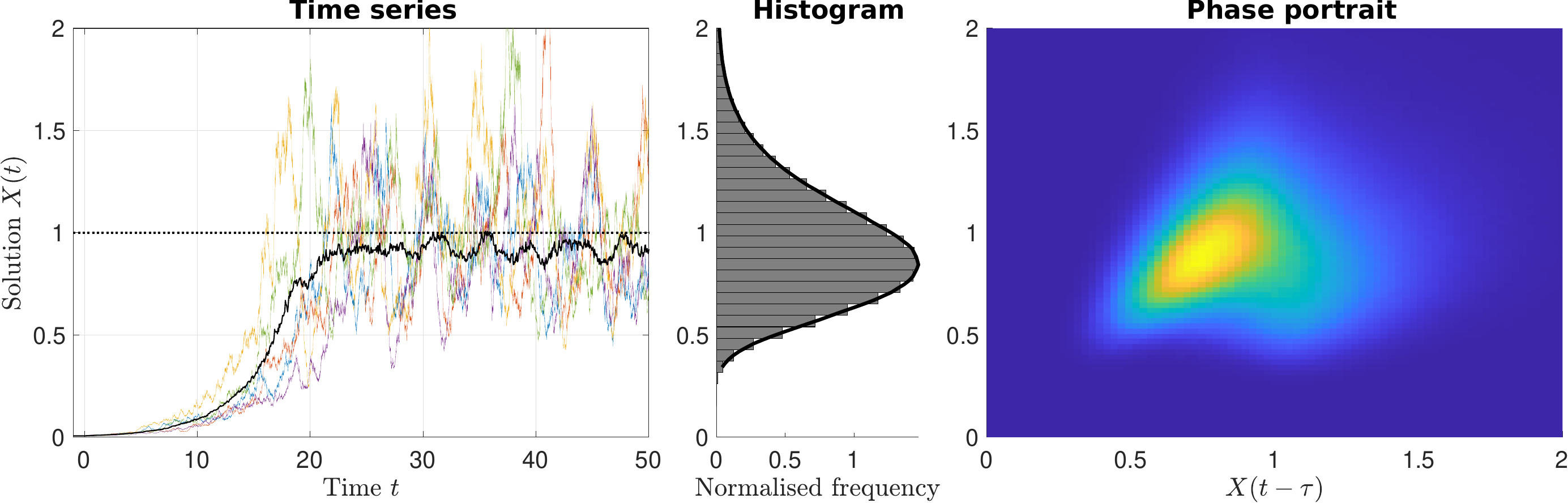}
    \includegraphics[width=0.99\linewidth]{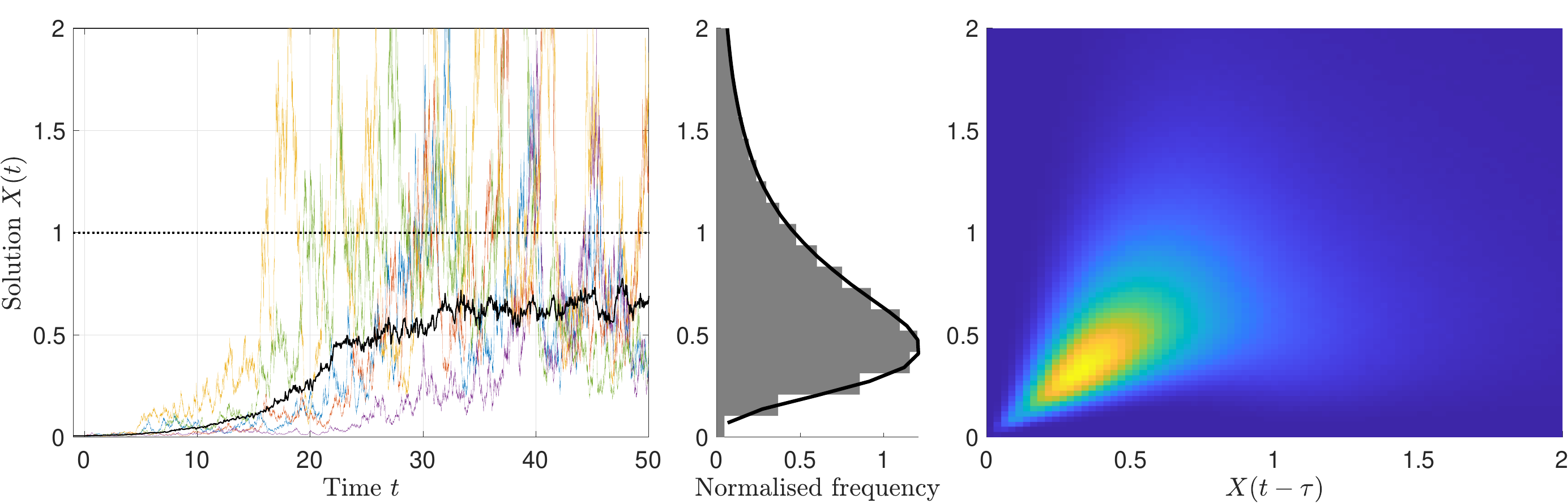}
    \caption{The same setting as in Figure \ref{fig:sigma:tiny}, but now with $\sigma=0.3$ (first row) and $\sigma=0.7$ (second row). Qualitatively, as the noise parameter $\sigma$ increases towards a certain threshold $\sigma_*>0$, the histogram and phase portrait tend to $\delta_0$ and $\delta_0\times \delta_0$, respectively. Differently put, the non-trivial measure becomes trivial once $\sigma$ is large enough.}
    \label{fig:sigma:small}
\end{figure}

\begin{figure}[!t]
    \centering
    \includegraphics[width=0.5\linewidth, trim = 0 .5cm 0 .45cm, clip]{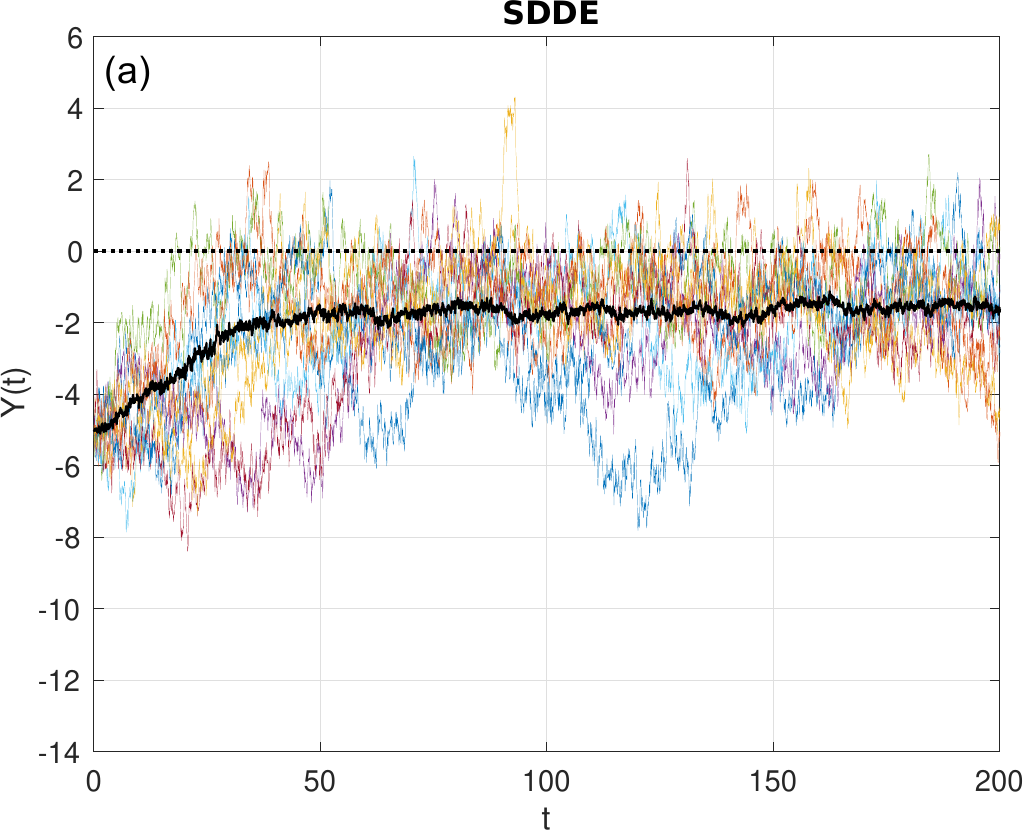}
    \,\includegraphics[width=0.4725\linewidth, trim = .88cm .5cm 0cm .45cm, clip]{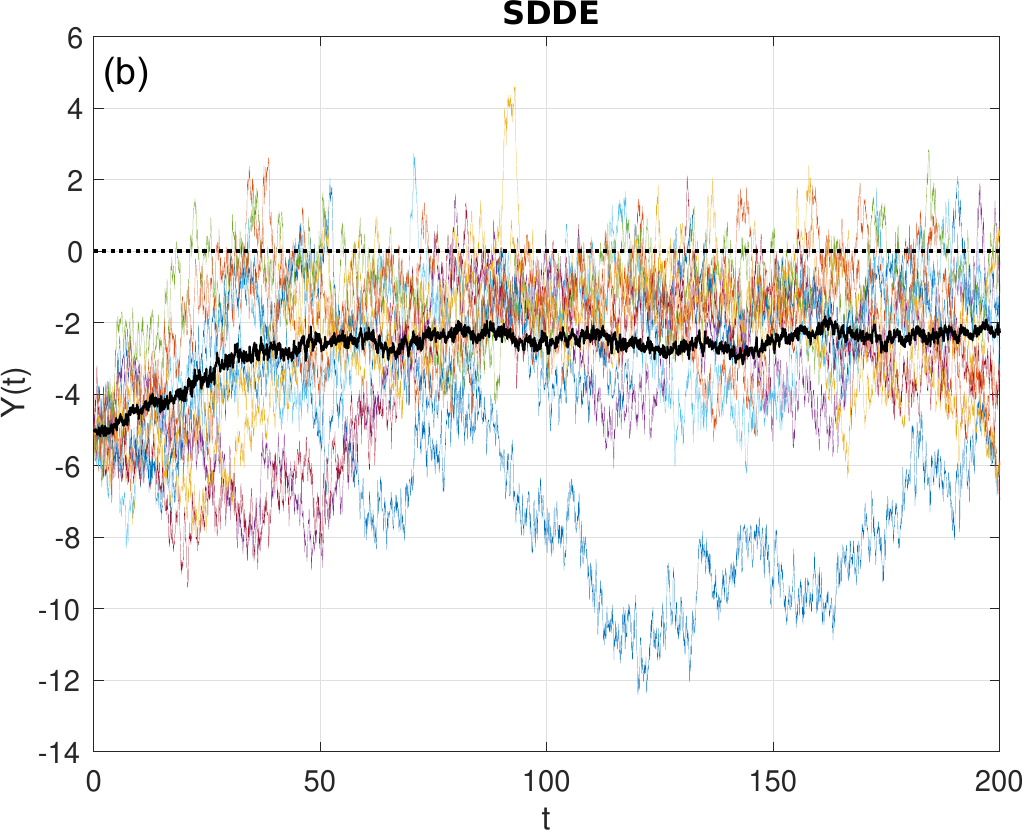}\\\vspace{.2cm}

    \includegraphics[width=0.5\linewidth, trim = 0 0cm 0 .45cm, clip]{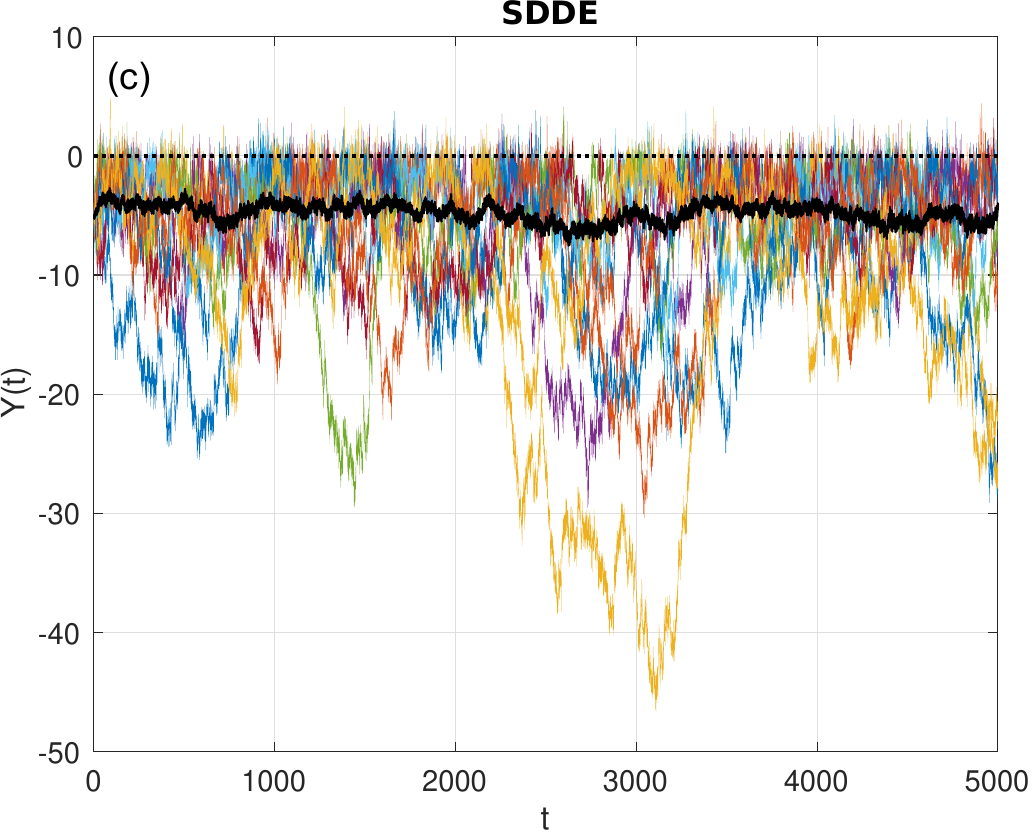}
    \,\includegraphics[width=0.4825\linewidth, trim = .88cm 0.0cm 0cm .45cm, clip]{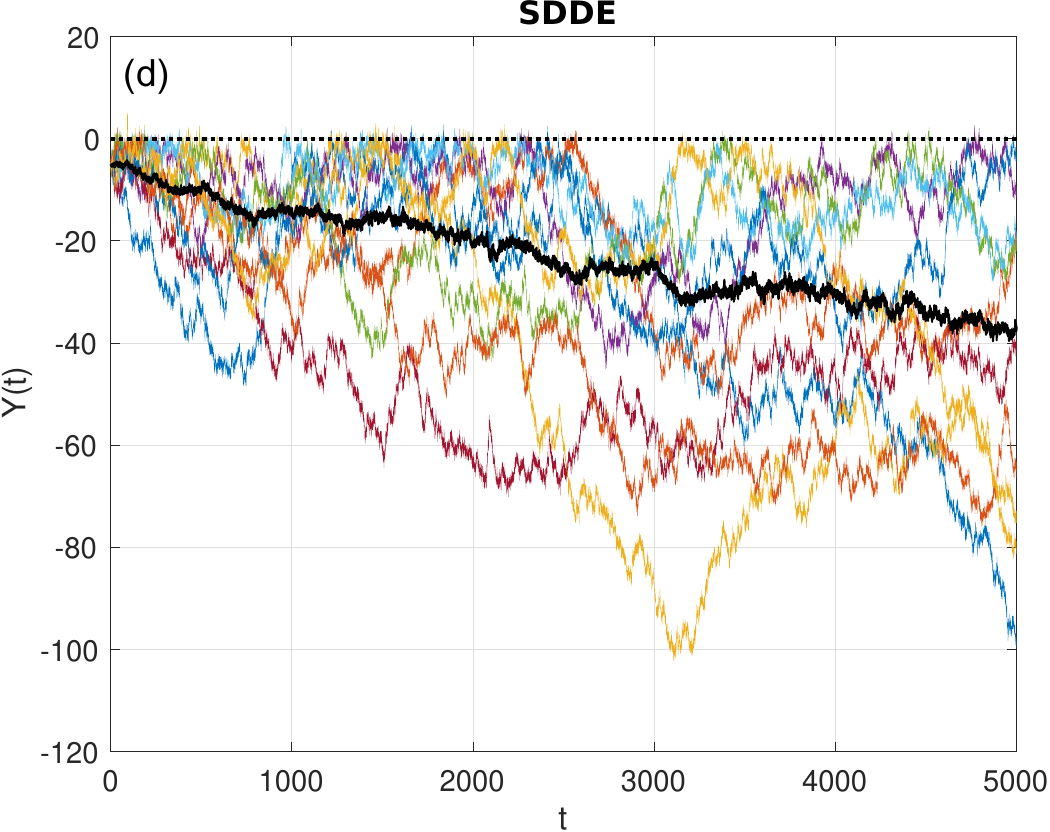}
    \caption{Simulation of  100 sample paths (only showing 20) of the solution to  equation \eqref{eq:log}  on the time interval $[0,5000]$, for the parameters $\gamma=0.5,$ $r=1,$ $n=30$, 
$f(x)=x(1+x^n)^{-1}$, as before, and (a) $\sigma=1.1$, (b) $ \sigma=1.2$, (c) $\sigma=1.3$ and (d) $\sigma=1.4.$   The black curves are the median of 100 sample paths. The (a) and (b) plots are zoomed in on the interval $[0,200]$. These plots suggest  that the solution is bounded below in probability for $\sigma\leq §1.2$. Notice that $\sigma=1$ is the threshold according to \eqref{eq:condition}. The (d) plot suggests that solutions are no longer bounded in probability for $\sigma\geq1.4.$}
    \label{fig:sigma>1}
\end{figure}


It is interesting to investigate the invariant measures obtained via the Krylov–Bogoliubov method (see Corollary \ref{cor:invariant-existence})
numerically. While these invariant measures are supported on an infinite dimensional ambient space, invariant measures with finite dimensional support can easily be observed. These arise as projections onto finite dimensional subspaces of the phase space \cite{vdBMG2026,vdBWright2026}. In Figures \ref{fig:sigma:tiny} and \ref{fig:sigma:small}, we specifically illustrate pushforwards of the invariant measure on $C([-\tau;0];(0,\infty)),$ corresponding to the projections $\varphi\mapsto \varphi(0)$ and $\varphi\mapsto(\varphi(-\tau),\varphi(0))$. These are depicted in the histogram and phase portrait, respectively. 

Theoretically, we know that solutions to equation \eqref{eq:stoch-mg-intro} with the earlier prescribed parameters are bounded in probability (Theorem \ref{thm:main}), hence induce non-trivial invariant measures (Corollary \ref{cor:invariant-existence}), for all $\sigma< 1$. Figure \ref{fig:sigma>1} demonstrates that, eventually, solutions are indeed no longer bounded away from zero in probability; solutions go extinct. Yet, this can happen far beyond the threshold $\sigma=1$|see condition \eqref{eq:condition}|and is the subject of Section \ref{sec:sharpness}.

The proof of Theorem \ref{thm:main} establishes the estimates in
Theorem~\ref{thm:negative-moment} for every prescribed $p>0$ such that
\begin{equation}
 r\kappa>\gamma+\frac{p+1}{2}\sigma^2.
\end{equation}
For $\sigma>0$, the admissible range for negative moments is therefore
\begin{equation}
 0<p<\frac{2(r\kappa-c)}{\sigma^2},
 \qquad c=\gamma+\frac12\sigma^2.
\end{equation}
For $\gamma$ and $\sigma$ fixed, these exponents become arbitrarily
large as $r\kappa\to\infty$. If $\sigma=0$, every $p>0$ is
admissible whenever $r\kappa>\gamma$. For, e.g., $\gamma=0.1, r=1, $ and $\sigma=0.1,$ we allow for $p<90$. Similar conclusions can be drawn for the positive moments, referring to Theorem \ref{thm:positive-moment}.

\subsection{Conditioning on the initial history}\label{sec:history}

In the moment estimates of Section \ref{sec:proofs}, we condition on $\mathcal F_0=\sigma(\Psi)$ to avoid imposing moment assumptions on the random initial history. It further allows us to investigate the behavior of  solutions as if the initial conditions were deterministic.
Conditional moments are moments for
the corresponding deterministic history. Lastly, conditional probability estimates give unconditional
ones simply by taking expectations \cite[Ch.~9]{Williams1991}.

In more detail, consider the natural
filtration 
\begin{equation}
 \mathcal F_t
 =\sigma(\Psi)\vee\sigma(W(s):0\le s\le t),
 \qquad
 \mathcal F_0=\sigma(\Psi),
\end{equation}
where $\sigma(\Psi)=\sigma(\Psi(\theta):-1\le\theta\le0)$.
Since $\Psi$ and $W$ are independent, conditioning on
$\mathcal F_0$ fixes the initial history and averages over
the subsequent Brownian noise. For a deterministic history, these
conditional expectations reduce to unconditional expectations.
Writing $X^\varphi$ for the solution
with deterministic history $\varphi\in C([-1,0];[0,\infty))$,  let
\begin{equation}
 u_{t,q}(\varphi):=\EE[(X^\varphi(t))^q].
 \end{equation}
 For random initial functions $\Psi$, we find the identity
 \begin{equation}
 \EE[X(t)^q\mid\mathcal F_0]
 =u_{t,q}(\Psi),
 \qquad q\in\R,
\end{equation}
almost surely; the follows from the Doob--Dynkin lemma.
Since for every fixed $\varphi$,
\begin{equation}
 \EE\sup_{0\le s\le T}(X^\varphi(s))^q<\infty,
 \qquad T<\infty,\quad q\in\R,
\end{equation}
path continuity and dominated convergence  imply continuity of
$t\mapsto u_{t,q}(\varphi)$, almost surely.

Recall the notation in Theorem~\ref{thm:negative-moment}. If $\EE [M]<\infty$, equivalently $\EE[\Psi(0)^{-p}]<\infty$, the proof
of Theorem~\ref{thm:negative-moment} can instead be written for
unconditional moments. Taking the expectation in
\eqref{eq:block-contraction} on both sides gives
\begin{equation}
 v_{n+1}\le qv_n+D,\qquad v_n:=\EE[X(nT)^{-p}].
\end{equation}
Iterating this recurrence and using \eqref{eq:voc-lower} between
successive multiples of $T$, similarly yields
\begin{equation}
  \EE[X(t)^{-p}]
 \le C+q^{\lfloor t/T\rfloor}(\EE [M]-C)_+
 \le\max\{C,\EE [M]\},\qquad t\ge0.
\end{equation}
Markov's inequality  gives boundedness away from zero in
probability. The positive-moment argument likewise admits an
unconditional formulation when the corresponding moment of $\Psi(0)$ is finite.

Notice that the initial moment assumption is unnecessary for boundedness in
probability itself, cf.   \cite[Prop. 3.5]{vdBInvariant2026} and \cite[Prop. 5.5]{vdBMG2026}. Indeed, \eqref{eq:negative-uniform} gives, for
$K\ge C$ and $\varepsilon>0$,
\begin{equation}
 \sup_{t\ge0}\PP\{X(t)\le\varepsilon\}
 \le\PP\{M>K\}+K\varepsilon^p.
\end{equation}
Since $M<\infty$ almost surely, choose $K$ large and then
$\varepsilon$ small. Conditional estimates allow us to get to the same conclusion, namely by means of the dominated convergence theorem; see the proofs of Theorems \ref{thm:negative-moment} and \ref{thm:positive-moment}.

Although the (conditional) moments generally depend on the entire history,
the bounds in Theorems~\ref{thm:negative-moment}
and~\ref{thm:positive-moment} depend on it only through $\Psi(0)$.
Indeed, nonnegativity of $f$ gives
$X(t)\ge\Psi(0)\Phi(t)$, while boundedness of $f$ controls the delayed
recruitment independently of the preceding history. Thus different
histories with the same endpoint can produce different transients while
satisfying the same moment bounds.

\subsection{Persistence versus extinction: how sharp is the threshold?}
\label{sec:sharpness}
Lyapunov spectra  for linear stochastic functional differential equations
are treated in \cite{MohammedScheutzow1996,MohammedScheutzow1997}; see also \cite{GhaniVarzanehRiedelScheutzow2022,SuZeng2026} for
more recent results. In particular, existence of the first Lyapunov exponent is known for the linearization. We will now give a detailed description of how the first Lyapunov exponent varies with respect to the parameters.
\begin{proposition}[First Lyapunov exponent]
\label{prop:linear-lyapunov}
For every $b\ge0$, the fundamental solution of
\eqref{eq:linear-fundamental} has a finite deterministic first Lyapunov
exponent $\lambda(b)$ satisfying, almost surely,
\begin{equation}
\label{eq:linear-log-growth}
 \lambda(b)=\lim_{t\to\infty}\frac{\log G_b(t)}t
           =\lim_{n\to\infty}\frac{\log G_b(n)}n
           =\sup_{n\ge1}\frac{\EE\log G_b(n)}n.
\end{equation}
The map $b\mapsto\lambda(b)$ is continuous and strictly increasing on
$[0,\infty)$, and $u\mapsto\lambda(\e^u)$ is convex. With
$c:=\gamma+\sigma^2/2$, one has
\begin{equation}
 \lambda(0)=-c,\qquad
 -c\le\lambda(b)\le-c+b\e^{\gamma+\sigma^2},\qquad b\ge0.
\end{equation}
\end{proposition}

\begin{proof}
Lemma~\ref{lem:finite-time-log-moments} gives
$\log G_b(T)\in L^1(\mathbb P)$ for every $T>0$. By
Proposition~\ref{prop:wellposed}, $G_b/\Phi$ is nondecreasing, so
\begin{equation}
 \frac{G_b(t-1)}{G_b(t)}\le\frac{\Phi(t-1)}{\Phi(t)},\qquad t\ge1.
\end{equation}
It\^o's formula gives, for $n\ge1$,
\begin{equation}
 \log G_b(n)=-cn+\sigma W(n)
       +b\int_1^n\frac{G_b(s-1)}{G_b(s)}\,ds.
\end{equation}
Since $\EE[\Phi(s-1)/\Phi(s)]=\e^{\gamma+\sigma^2}$, taking expectations
yields
\begin{equation}
 -cn\le\EE\log G_b(n)
 \le-cn+b\e^{\gamma+\sigma^2}(n-1).
\end{equation}

The representation in Lemma~\ref{lem:linear-representation} and
positivity give $G_b(t+s)\ge G_b(t)G_b^{(t)}(s)$. Hence, the array
\begin{equation}
 A_{m,n}:=-\log G_b^{(m)}(n-m),\qquad 0\le m<n,
\end{equation}
satisfies $A_{\ell,n}\le A_{\ell,m}+A_{m,n}$ for integers
$0\le\ell<m<n$. On the Wiener path space, the Brownian shift
$(\theta_kW)(t):=W(k+t)-W(k)$ satisfies
$A_{m+k,n+k}=A_{m,n}\circ\theta_k$; see, e.g., \cite[Sec.~A.3, pp.~547--548]{Arnold1998} or \cite[Example 1.1.7]{Chueshov2002}. Thus, the array is stationary under diagonal
translations of its indices: its joint distribution remains
unchanged when $(m,n)$ is replaced by $(m+k,n+k)$.
Moreover, the Brownian shift $\theta_1$ is ergodic, because the unit-length Brownian path increments
are independent and identically distributed. Moreover,
\begin{equation}
 \EE |A_{0,1}|<\infty,\qquad
 \EE [A_{0,n}]\ge\bigl(c-b\e^{\gamma+\sigma^2}\bigr)n.
\end{equation}
Kingman's subadditive ergodic theorem \cite{Kingman1973} therefore gives
a deterministic limit
\begin{equation}
 \lim_{n\to\infty}\frac{A_{0,n}}n
 =\inf_{n\ge1}\frac{\EE [A_{0,n}]}n,
 \qquad\text{almost surely}.
\end{equation}
Changing signs proves the integer-time limit and supremum formula in
\eqref{eq:linear-log-growth}; the expectation bounds above give the
stated bounds on $\lambda(b)$. For $n\le t\le n+1$, monotonicity gives
\begin{equation}
 \frac{G_b(n)}{\Phi(n)}
 \le\frac{G_b(t)}{\Phi(t)}
 \le\frac{G_b(n+1)}{\Phi(n+1)}.
\end{equation}
Since $\log\Phi(t)/t\to-c$ almost surely, this extends the limit to
continuous time.

Iterating the integrating factor equation over successive unit
intervals writes $G_b(n)/\Phi(n)$ as a polynomial in $b$ with
nonnegative coefficients. Consequently, $u\mapsto\log G_{\e^u}(n)$ is convex
and nondecreasing, and \eqref{eq:linear-log-growth} shows that
$g(u):=\lambda(\e^u)$ has the same properties. The bounds above give
$\lim_{u\to-\infty}g(u)=\lambda(0)=-c$. For $b>0$,
$G_b(2)>\Phi(2)$, so
\begin{equation}
 \lambda(b)\ge\tfrac12\EE\log G_b(2)>-c.
\end{equation}
A nondecreasing convex function lying strictly above its limit at
$-\infty$ is strictly increasing. Together with the bounds at $b=0$,
these observations prove continuity and strict monotonicity on
$[0,\infty)$.
\end{proof}

\begin{proposition}
\label{prop:sharp-persistence}
If $\lambda(r\kappa)>0$, every solution with nonnegative initial history
$\Psi$ and $\Psi(0)>0$ almost surely
is bounded away from zero in probability uniformly in time.
If, in addition, $f$ is bounded, a non-trivial invariant measure exists.
If
\begin{equation}
\label{eq:sublinearity}
 f(x)\le\kappa x\qquad(x>0),
\end{equation}
and $\lambda(r\kappa)<0$, then every such solution converges to zero
almost surely.
\end{proposition}

\begin{proof}
If $\lambda(r\kappa)>0$, \eqref{eq:linear-log-growth} gives $T$ with
$\EE\log G_{r\kappa}(T)>0$. Finite-time continuity permits
$b<r\kappa$ with the same property, so
Theorem~\ref{thm:negative-moment} gives persistence. When $f$ is bounded, one obtains the invariant measure similarly to the discussion prior to
Corollary~\ref{cor:invariant-existence}.

Under \eqref{eq:sublinearity}, put $b=r\kappa$. By
Proposition~\ref{prop:wellposed},
$C:=\max_{0\le s\le1}X(s)/G_b(s)<\infty$ almost surely.
Applying Lemma~\ref{lem:linear-representation} from time $1$ to
$CG_b-X$, whose history and forcing
$bX(t-1)-rf(X(t-1))$ are nonnegative, gives
$0<X(t)\le CG_b(t)$ for $t\ge1$.
Proposition~\ref{prop:linear-lyapunov} gives the continuous-time limit
in \eqref{eq:linear-log-growth}.
Consequently,
\begin{equation}
 \limsup_{t\to\infty}\frac{\log X(t)}t
 \le\lambda(r\kappa)<0
 \qquad\text{almost surely},
\end{equation}
which proves extinction.
\end{proof}

For feedbacks satisfying \eqref{eq:sublinearity} with $\kappa>0$,
including Mackey--Glass and Nicholson feedbacks with $\kappa=f'(0)=1$,
Proposition~\ref{prop:sharp-persistence} identifies the nonlinear
threshold with that of the linear comparison equation. We thus have a unique critical rate satisfying
\begin{equation}
 \lambda(r_{\mathrm c}\kappa)=0,
 \qquad 0<r_{\mathrm c}\le\frac{\gamma+\sigma^2/2}{\kappa}.
\end{equation}
Persistence holds for $r>r_{\mathrm c}$, and almost-sure extinction for
$r<r_{\mathrm c}$.


\begin{acknowledgments}
Both AI tools
Claude (Anthropic) and ChatGPT (OpenAI) have been used during the preparation of this
manuscript. The authors have reviewed all AI-assisted content and take full responsibility
for the final manuscript. 
\end{acknowledgments}
\hypersetup{linkcolor=black,citecolor=black,urlcolor=black}

\end{document}